\documentclass[twoside]{article}

\usepackage{amsthm}
\usepackage{amssymb}
\usepackage{amsmath}
\usepackage{amsfonts}
\usepackage{epic}
\usepackage{epsfig}
\usepackage{bm}
\usepackage{amscd}
\usepackage{graphics}
\usepackage{euscript}
\usepackage{color}

\font\Bbb=msbm10

\def\P{\hbox{\Bbb P}}
\def\C{\hbox{\Bbb C}}
\def\R{\hbox{\Bbb R}}
\def\Z{\hbox{\Bbb Z}}
\def\N{\hbox{\Bbb N}}

\newcommand{\scEtilde}{\ensuremath{\mathcal{E}}}
\newcommand{\scFtilde}{\ensuremath{\mathcal{F}}}
\newcommand{\scE}{\ensuremath{\mathcal{E}}}
\newcommand{\scF}{\ensuremath{\mathcal{F}}}
\newcommand{\scEbar}{\overline{\scE}}
\newcommand{\scFbar}{\overline{\scF}}

\newcommand{\scO}{\ensuremath{\mathcal{O}}}

\newtheorem{theorem}{\bf Theorem}
\newtheorem{lemma}[theorem]{\bf Lemma}
\newtheorem{remark}[theorem]{\bf Remark}

\newtheorem{proposition}[theorem]{\bf Proposition}

\newtheorem{definition}[theorem]{\bf Definition}

\numberwithin{equation}{section} \numberwithin{theorem}{section}

\author{Susumu TANAB\'E}               
\title{On monodromy representation of  period integrals
associated to an algebraic curve with bi-degree {\it (2,2)}}               
\begin{document}
\maketitle

\begin{abstract}
We study a problem related to Kontsevich's homological mirror symmetry conjecture for the case of a generic curve 
 $\cal Y$ with bi-degree {\it (2,2)} in  a product of projective lines $\P^{1} \times \P^{1}$.  We calculate two differenent monodromy representations  of period
integrals  for the affine variety ${\cal X}^{(2,2)}$ obtained by the dual polyhedron mirror variety construction from $\cal Y$. The first method that gives a full representation of the fundamental group of the complement to singular loci relies on the generalised Picard-Lefschetz theorem. 
The second method uses the analytic continuation of the Mellin-Barnes integrals that gives us a proper subgroup of the monodromy group.
It turns out both representations admit a Hermitian quadratic invariant form that is given by a Gram matrix  of  a split generator of the derived category of coherent sheaves on on $\cal Y$ with respect to the Euler form. 
\end{abstract}



\setcounter{section}{-1}
\section{Introduction \label{sec:introduction }}

In this note we study a problem related to Kontsevich's homological mirror symmetry conjecture for the case of a generic curve 
 $\cal Y$ with bi-degree {\it (2,2)} in  a product of projective lines $\P^{1} \times \P^{1}$.

In \cite{TaU13} we studied 
 the Calabi-Yau complete intersection $Y$ in a weighted projective space.
We claimed  
 that the space of Hermitian quadratic invariants of
the monodromy group for the period integrals associated to  the Batyre-Borisov mirror dual complete intersection $X$
is one-dimensional,
and spanned by the Gram matrix of a split-generator
of the derived category of coherent sheaves on $Y$ with respect to the Euler form.
To show this result  the monodromy group has been calculated as  monodromy group for  Pochhammer hypergeometric functions.  

In following the spirit of  \cite{TaU13} where period integrals depending  on single deformation parameter are studied, we  establish a similar result for the case of 
period integrals depending on two variables. In particular, here the crucial moment is an interpretation of period integrals as Horn hypergeometric functions in two variables whose rank 4 monodromy representation is reducible.

Namely we consider
the generic curve $\cal Y$ of bi-degree (2,2) in $\P^1 \times \P^1$ 
and its mirror counter-part  ${\cal X}^{(2,2)}_{x,y}$ obtained by Batyrev's dual polyhedron construction (\ref{X22}). 
We establish the following result on the monodromy representation of  period integrals defined as integrals along cycles from
 $H_1({\cal X}^{(2,2)}_{x,y}).$ In this note we shall use both notations $\sqrt -1$ and $i$ to denote the unit pure imaginary number. 

\begin{theorem}
The monodromy representation $H_0$ calculated by the Mellin-Barnes integrals
(Proposition \ref{rank4monodromy})
as well as that obtained by the generalised Picard-Lefschetz theorem ( 
Proposition \ref{gotomatsumoto})
 admits a Hermitian quadratic invariant $ \sqrt -1 {\sf G}$
for 
$$
{\sf G}= \left(
\begin{array}{cccc}
 0 &- 2 & 0 & 2 \\
 2 & 0 & -2 & 0 \\
 0 & 2 & 0 & -2 \\
 -2 & 0 & 2 & 0
\end{array}
\right),
$$
up to conjugate isomorphism of representations.
Here the anti-symmetric matrix ${\sf G}$ is
a Gram matrix with respect to the Euler form of a split generator on $\cal Y$
 obtained by restricting 
a full exceptional collection 
$({\mathcal F}_i)_{i=1}^4$ determined by  (\ref{eq:Euler_form}) that is a right dual exceptional collection to
$({\mathcal O}, {\mathcal O}(1,0), {\mathcal O}(1,1), {\mathcal O}(2,1))$
on $D^b Coh\; (\P^1 \times \P^1) $ restricted
to $\cal Y$.
 \label{th:invariant}
\end{theorem}

Our theorem \ref{th:invariant} is closely related
to the works of Horja \cite[Theorem 4.9]{Horja}
and Golyshev \cite[\S 3.5]{Gol},
which originated from a conjecture proposed by Kontsevich in 1998.

The main difference of \cite{TaU13} from the works  \cite{Horja}, \cite{Gol}  lies in the fact that
it treats the reducible system 
which contains sections
not coming from period integrals on the compact mirror manifold.
In the case of the irreducible local system (hypergeometric equation), 
Golyshev gave a beautiful interpretation
in terms of autoequivalences of the derived category
of the mirror manifold.

Our proof of Theorem \ref{th:invariant}   relies on calculus of a Horn hypergeometric system wıth reducible monodromy,
just as in   \cite{TaU13} where the case of the irreducible hypergeometric system has been extended to that of a reducible system.

We shall recognise that our description of the representation $H_0$ in
Proposition \ref{rank4monodromy}  is not conclusive so far as we ignore its nature as a representation of  the fundamental group of the complement to singular loci
of the Horn hypergeometric system. Furthermore the representation $H_0$ gives only a proper subgroup of the entire monodromy group ( Proposition \ref{rank4monodromy}, Remark \ref{remark:gamma2}) .    None the less it admits a one dimensional real vector space of Hermitian quadratic  forms.

The core part of this note is the monodromy calculus in Proposition \ref{rank4monodromy} made by means of analytic continuation of Mellin-Barnes integrals. To our knowledge no trial has been made to calculate a global monodromy representation of bivariate period integrals using Mellin-Barnes integrals.
We shall, however, mention \cite[4.3]{Horja} as one of precious testimony where this approach was successfully applied to a problem related to the Kontsevich's homological mirror conjecture. The proposal made in \cite{Beukers} also deserves special attention for further studies of period integrals as a class of A-hypergeometric functions.

One of advantages of our method consists in the fact that the choice  of the solution basis 
(\ref{ex(1,0)(0,1)(2,2)}) allows us to calculate the monodromy without  connection matrices.  In the calculus of the monodromy of univariate hypergeometric functions (\cite[2.4.6]{IKSY}, \cite{Smith}) solution basis has been chosen in dependence on the asymptotic behaviour (i.e. characteristic exponents) of the solution around singular points and quite involved calculus of connection matrix was necessary.
In this note every data on the monodromy are calculated relying exclusively on the Mellin transform
(\ref{oresato}) that can be easily derived from the Newton polyhedron $\Delta_{F_{2,2}}$ of  the Laurent polynomial  ${F_{2,2}}$ (\ref{polyhedron1})  according to the principle proposed in \cite{Tan07}. After this principle the Mellin transform of a period integral has poles with a semi-group like structure
whose features are determined by outer normals to the faces of $\Delta_{F_{2,2}}$ and their scalar product with exponent characterising the monomial cohomology class present in the integrand.

The author expresses his gratitude to Kazushi  Ueda who furnished the concrete form of
the Gram matrix  ${\sf G}$ upon his  request. Without this information it would have been impossible to make any kind of trial.
His acknowledgement goes also to M.Uluda\u{g}, F.Beukers, Y.Goto, J.Kaneko for valuable discussions and comments. 
A special recognition goes out to the organisers of the First Romanian-Turkish Mathematics Colloquium at Constan\c{t}a in October 2015.

\section{Preliminaries on  elliptic integrals and Gauss hypergeometric functions \label{sec:preliminaries}}

First of all we recall basic facts on the relation between period integrals for the elliptic curve (elliptic integrals)
and Gauss hypergeometric functions.

 Consider a double covering of $\P^1 \setminus \{x=0,1,t,\infty \}$
\begin{equation}
{\mathcal R}=\{(x,y) \in \P^2 : y^2=x(x-1)(x-t) \}. 
\end{equation}

It is known that this algebraic curve (elliptic curve)  gives a Riemann surface of genus $1.$ 
One can define the elliptic integral for a cycle $\alpha \in H_1(\mathcal R)$
\begin{equation}
\displaystyle{\int_\alpha \frac{dx}{y}} = \displaystyle{\int_\alpha \frac{dx}{\sqrt{x(x-1)(x-t)}}} 
\end{equation}

 \begin{figure}[h]
\begin{center}
\includegraphics[width=0.4\textwidth]{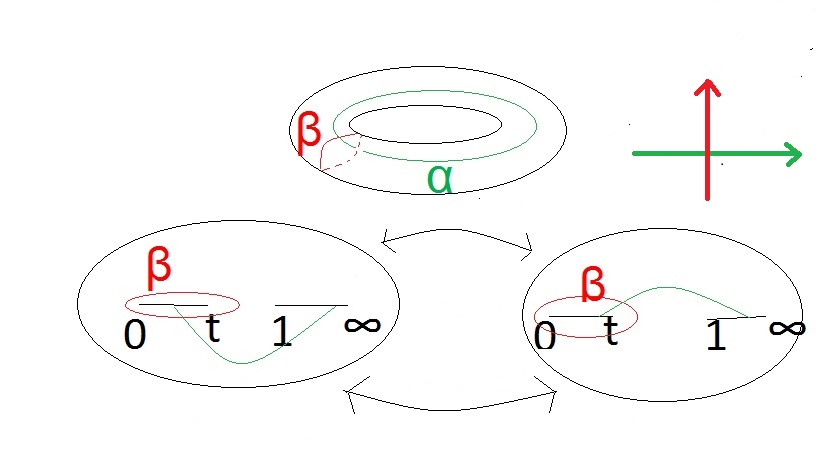}
\caption{  Cycles on the curve $R$.}
\end{center}
\label{fig:pathpicture1}
\end{figure}

\noindent
This integral can be expressed by the classical Gauss hypergeometric function 
\begin{equation}
F(\frac{1}{2},\frac{1}{2}, 1|t)=\sum_{m=0}^{\infty} \frac{\Gamma(\frac{1}{2} +m)^2}{(m!)^2}t^m,
\label{solutionHGF}
\end{equation}
 for $|t|<1$ and it satisfies a second order differential equation

\begin{equation}
 [(\theta_{t})^2 -t(\theta_{t}+ \frac{1}{2})^2]\displaystyle{f(x)=0},
\label{GHGeq}
\end{equation}
$(\theta_{t}= t \frac{\partial}{\partial t}).$ 
The solution space to this equation has dimension 2 that is equal to the rank of $ H_1(\mathcal R).$ This means that a general solution to (\ref{GHGeq})  is given by $\displaystyle{\int_{n \alpha+ m \beta} \frac{dx}{y}}$ for some
$(n,m) \in \Z^2.$
 \vspace{1pc}

\noindent
We remark here that
the solution (\ref{solutionHGF}) admits a Mellin-Barnes integral representation (sum of residues)
\begin{equation}
\sum_{n=0}^\infty \displaystyle{Res_{z = n }\; (\frac{\Gamma ({\frac 1 2}+z)^2\Gamma (-z) (-t)^z dz}{\Gamma (1+z)})}, 
\end{equation}

\noindent
As a basis of the cohomology of the elliptic curve $H^1(\mathcal R)$ we can choose a couple of rational forms
\begin{equation}
\frac{dx}{y}, \;\;\;\;\;\;x\frac{dx}{y}
\end{equation}
The dimension of the $\C$ vector space
$H^1(\mathcal R)$ is  equal to $ 2 = rank H_1(\mathcal  R).$   

The period  integral
$$\displaystyle{\int_{n \alpha+ m \beta} \frac{x dx}{y}},\;\;\; (n,m) \in \Z^2$$
also satisfies a Gauss hypergeometric equation analogous to (\ref{GHGeq}) ,
\begin{equation} [(\theta_{t})^2 -t(\theta_{t} - \frac{ 1}{2})^2] u(t)=0.  
\end{equation}

The monodromy group $G$ of a solution system to (\ref{GHGeq})
admits the following representation (\cite{BeukersHeckman})
$$G \subset SL(2,\Z) = Sp(1,\Z)$$
\begin{equation} G= \langle h_0:=\left(
\begin{array}{cc}
0 & 1 \\
-1 & 2 \\
\end{array} \right), h_\infty:=\left(
\begin{array}{cc}
-2 & -1 \\
1 & 0 \\
\end{array} \right)\rangle
\end{equation}

By conjugation with the matrix $C_0 \in SL(2, \R) $
$$C_0=\left(
\begin{array}{cc}
 \frac{1}{\sqrt{2}} & \frac{1}{\sqrt{2}} \\
 -\frac{1}{\sqrt{2}} & \frac{1}{\sqrt{2}}
\end{array}
\right),$$
we get the following  two matrices
$$  h_0^{C_0}= \left(
\begin{array}{cc}
 1 & 0 \\
 -2 & 1
\end{array}
\right),    h_\infty^{C_0}=\left(
\begin{array}{cc}
 -1 & 2 \\
 0 & -1
\end{array}
\right)$$
that generates together with $-Id_2$ (that becomes trivial in passing to the projective linear group),
$$ \Gamma(2) = \{g \in SL(2,\Z);    g \equiv Id_2 \;\; mod \;2\}$$
the principal congruence subgroup of level 2.  From now on we use the notation
$$ A^B = B^{-1} A B$$
for $A \in M(m, \C)$ and $B \in GL(m,\C),$ $m \geq 1.$
Namely $ h_0^{C_0} = C_0^{-1} h_0 C_0$ etc. 

The intersection matrix $Int$ with respect to   the basis $\alpha,\beta$ of $H_1(\mathcal R)$
\vspace{1pc}
$$
Int:=\left(
\begin{array}{cc}
<\alpha,\alpha>& <\alpha,\beta> \\
<\beta,\alpha> & <\beta,\beta> \\
\end{array} \right) = \left(
\begin{array}{cc}
0 & 1 \\
-1 & 0 \\
\end{array} \right),  $$
\begin{equation}^th_0.Int.h_0 = Int,\;\;\;\; ^th_\infty.Int.h_\infty = Int. \end{equation}
The intersection matrix $Int$ is the simplest example of the Hermitian quadratic  invariant associated to a hypergeometric functions/period integrals (see \cite[Chapter 4]{BeukersHeckman}).

Here we shall remark that the conjugate matrix $C_0$ satisfies $^tC_0. Int. C_0 = Int$
and the monodromy representation $G$ can be determined only up to a conjugate by a matrix
of $Sp(1,\R) = \{g \in GL(2, \R);   ^tg.Int.g= Int \} .$
This kind of ambiguity will play essential r\^ole as we compare different presentations of a monodromy group. 

In the remaining part of the note
all statements mentioned in this section will be generalised to the case of a bi-degree (2,2) curve. 

\section{ Period integrals of a bi-degree (2,2) curve}

The generic curve $\cal Y$ with bi-degree  $(2,2)$ in $\P^{1} \times \P^{1}$ is defined by
a Laurent polynomial whose Newton polyhedron is
\begin{equation}
 \{(\alpha,\beta) \in \R^2; -1 \leq \alpha \leq 1, -1 \leq \beta \leq 1    \}. 
\label{dualpolyhedron2}
\end{equation}

\rm
The main object of this article is an affine curve
\begin{equation}
 {\cal X}^{(2,2)}_{x,y}= \{(z,w);  F_{2,2}(z,w)=0 \}  = {\rm elliptic\; curve} \setminus 3 \; points
\label{X22}
\end{equation}
for
\begin{equation}
 F_{2,2} (z,w)=1 + z+\frac{x}{z} + w+\frac{y}{w}
\label{polyhedron1}
\end{equation}
 whose Newton polyhedron is defined as
the dual polyhedron to (\ref{dualpolyhedron2}) after Batyrev's construction.
The period integral associated to the curve (\ref{X22}) is defined as
$$  I_{a,b}(x,y)= \int_{\gamma} \frac{z^a w^b}{z w }\frac{ dz \wedge dw}{dF_{2,2}}$$
for $\gamma \in H_{1}  ( {\cal X}^{(2,2)}_{x,y}  )$ and a monomial $z^a w^b \in \C[z,w].$

After the method in \cite{Tan07} we calculate the Mellin transform of the period integral that equals to
$$ \Gamma(s+a)\Gamma(s)\Gamma(t+b) \Gamma(t) \Gamma(1 -a-b - (2s+2t))$$
up to multiplication by a meromorphic period function $\phi(s,t)$ such that $\phi(s+a', t+b') = \phi(s,t)$ for every $(a'.b') \in \Z^2.$
Thus the period integral $I_{a,b}(x,y)$ satisfies the following system of linear PDE,
\begin{equation}
\begin{array}{l}
\left(\theta_{x}(\theta_x + a)  - x (2\theta_x + 2\theta_y+1+a+b)(2\theta_x + 2\theta_y+2+a+b)\right)f(x,y)=0, \\
\left(\theta_{y}(\theta_y + b)  - y (2\theta_x + 2\theta_y+1+a+b)(2\theta_x + 2\theta_y+2+a+b)\right)f(x,y)=0.
\end{array}
\label{horn(1,0)(0,1)(2,2ab)}
\end{equation}
Further we use the notation $$ \theta_x = x \frac{\partial }{\partial x}, \theta_y = y \frac{\partial }{\partial y}.$$ 

This type of system of differential equations is called Horn hypergeometric system and solutions to it are called Horn hypergeometric functions (see \cite{DMS}, \cite{PST}, \cite{SadykovTanabe}).
The system \ref{horn(1,0)(0,1)(2,2ab)} has a solution holomorphic in the neighbourhood of $(x,y)$;
\begin{equation}
f^1_{1,1} (x,y)  = \sum_{(i_1, i_2 ) \in \Z^2_{\geq 0}} \frac{\Gamma(\frac{1-a-b}{2} + i_1 + i_2) \Gamma( \frac{2-a-b}{2} + i_1 + i_2)}{\Gamma(1-a+i_1) \Gamma(1-b+i_2)  (i_1!) (i_2!)} (4x)^{i_1} (4y)^{i_2}.
\label{abholsolution}
\end{equation}
As the rank $  H_{1}  (  {\cal X}^{(2,2)}_{x,y}) =4$
that is calculated by the area of a parallelogram with vertices $\{ (\pm 1,0), (0,\pm 1) \}$
(the Newton polyhedron of $F_{2,2} (z,w)$ for $(x,y) \in (\C^\ast)^2$)
we conclude that every solution to the system 
(\ref{horn(1,0)(0,1)(2,2ab)})
is a linear combination over $\C$ of period integrals.

In particular the period integral $I_{0,0}(s,y)$ satisfies 
 the Horn system with holonomic rank $4$ (see \cite[Corollary 4.3]{DMS}):
\begin{equation}
\begin{array}{l}
\left(\theta_{x}^2  - x (2\theta_x + 2\theta_y+1)(2\theta_x + 2\theta_y+2)\right)f(x,y)=0, \\
\left(\theta_{y}^2  - y (2\theta_x + 2\theta_y+1)(2\theta_x + 2\theta_y+2)\right)f(x,y)=0.
\end{array}
\label{horn(1,0)(0,1)(2,2)}
\end{equation}
In fact every period integral $I_{0,0}(x,y)$ can be expressed as residues of the Mellin transform that is known under the name of Mellin-Barnes integral
\begin{equation}
\int_{\Gamma_k} \phi(s,t) \Gamma(s)^2\Gamma(t)^2 \Gamma(1 - (2s+2t)) ds \wedge dt,
\label{oresato}
\end{equation}
where $\Gamma_k$ is one of the following pole lattices (points with a semi-group structure located inside of a cone) 
\begin{equation}
\begin{array}{l}
\Gamma_1= \{ (s,t) \in \C ; s \in \Z_{\leq 0 }, t \in \Z_{\leq 0} \},\\
\Gamma_2= \{ (s,t) \in \C ; t \in \Z_{\leq 0 }, 2s+2t-1 \in \Z_{\geq 0} \},\\
\Gamma_3= \{ (s,t) \in \C ; s \in \Z_{\leq 0 }, 2s+2t-1 \in \Z_{\geq 0} \}.
\end{array}
\label{poles}
\end{equation}

\begin{center}
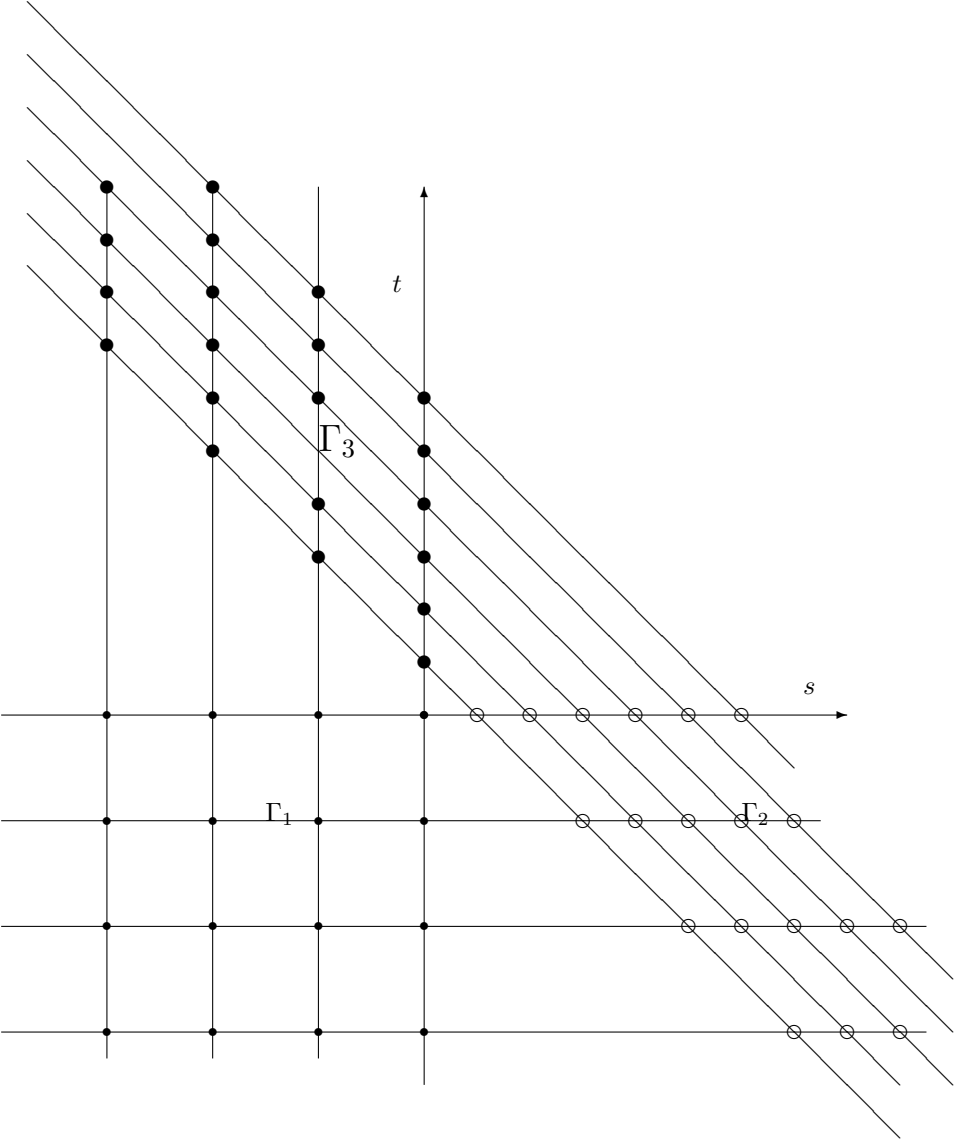
\begin{figure}[ht!]
\begin{minipage}{2cm}
\begin{picture}(400,400)
  \put(80,-80){\vector(0,1){340}}
  \put(-80,60){\vector(1,0){320}}
  \put(220,68){ $s$}
  \put(68,220){$t$}

  \put(-80,20){\line(1,0){310}}
  \put(-80,-20){\line(1,0){350}}
  \put(-80,-60){\line(1,0){350}}
\put(-70,230){\line(1,-1){330}}
\put(-70,250){\line(1,-1){330}}
\put(-70,270){\line(1,-1){350}}
\put(-70,290){\line(1,-1){350}}
\put(-70,310){\line(1,-1){350}}
\put(-70,330){\line(1,-1){290}}   
\put(40,260){\line(0,-1){330}}
\put(0,260){\line(0,-1){330}}
\put(-40,260){\line(0,-1){330}}
  
\put(80,60){\circle*{3}}
\put(80,80){\circle*{5}}
\put(80,100){\circle*{5}}
\put(80,120){\circle*{5}}
\put(80,140){\circle*{5}}
\put(80,160){\circle*{5}}
\put(80,180){\circle*{5}}
 \put(40,120){\circle*{5}}
 \put(40,140){\circle*{5}}
\put(40,160){\Large $\Gamma_3$}
 \put(40,180){\circle*{5}}
 \put(40,200){\circle*{5}}
 \put(40,220){\circle*{5}}
\put(0,160){\circle*{5}}
 \put(0,180){\circle*{5}}
 \put(0,200){\circle*{5}}
 \put(0,220){\circle*{5}}
 \put(0,240){\circle*{5}}
\put(0,260){\circle*{5}}
 \put(-40,200){\circle*{5}}
 \put(-40,220){\circle*{5}}
 \put(-40,240){\circle*{5}}
\put(-40,260){\circle*{5}}
\put(80,60){\circle*{3}}
\put(80,20){\circle*{3}}
\put(80,-20){\circle*{3}}
\put(80,-60){\circle*{3}}
\put(40,60){\circle*{3}}
 \put(40,-20){\circle*{3}}
 \put(40,-60){\circle*{3}}
 \put(40,20){\circle*{3}}
\put(20,20){$\Gamma_1$}
 \put(0,60){\circle*{3}}
 \put(0,-20){\circle*{3}}
 \put(0,-60){\circle*{3}}
 \put(0,20){\circle*{3}}
\put(-40,-60){\circle*{3}}
\put(-40,20){\circle*{3}}
\put(-40,-20){\circle*{3}}
\put(-40,60){\circle*{3}}
\put(100,60){\circle{5}}
\put(120,60){\circle{5}}
\put(140,60){\circle{5}}
\put(160,60){\circle{5}}
\put(180,60){\circle{5}}
\put(200,60){\circle{5}}
\put(140,20){\circle{5}}
\put(160,20){\circle{5}}
\put(180,20){\circle{5}}
\put(200,20){\circle{5}}
\put(200,20){$\Gamma_2$}
\put(220,20){\circle{5}}
\put(180,-20){\circle{5}}
\put(200,-20){\circle{5}}
\put(220,-20){\circle{5}}
\put(240,-20){\circle{5}}
\put(260,-20){\circle{5}}
\put(220,-60){\circle{5}}
\put(240,-60){\circle{5}}
\put(260,-60){\circle{5}}
\end{picture}
\end{minipage}
\vskip4cm \caption{Poles of $\Gamma(s)^{2}\Gamma(t)^{2}\Gamma(1-(2s+2t)) x^{-s}y^{-t}$}
\label{fig:MellinPoles}
\end{figure}
\end{center}
\vspace{2cm}

We remark that the affine part  $S$ of the singular loci of the system (~\ref{horn(1,0)(0,1)(2,2)})
 is given  by a parabola and two coordinate axes,
$$S =\{(x,y) \in \C^2 ; xy (16(x-y)^2 -8(x+y) +1)=0 \}.$$

\begin{figure}[h]
\begin{center}
\includegraphics[width=0.4\textwidth]{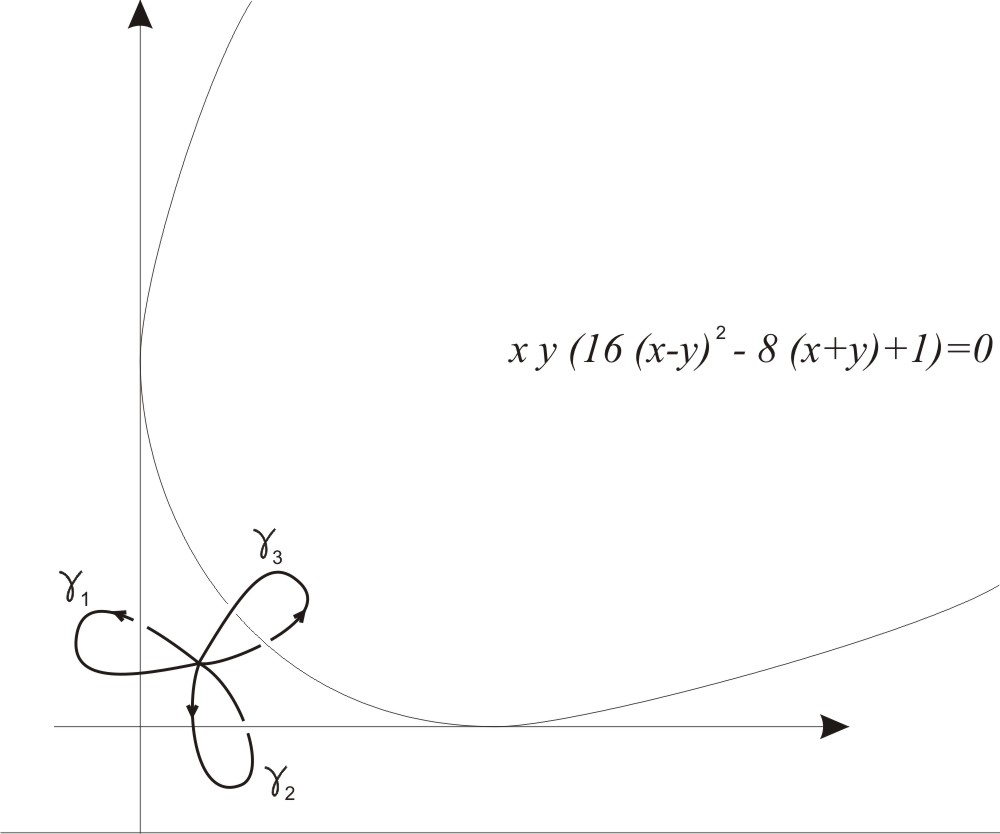}
\caption{generators of 
$\pi_1(\C^2 \setminus S)$ }
\label{fig:fundamentalgroup }
\end{center}
\end{figure}

 For  $S$ the following presentation of  the fundamental group has been established in \cite{Kaneko}, \cite{ATU}:
\begin{equation}
\pi_1(\C^2 \setminus S) =\left< \gamma_1, \gamma_2, \gamma_3 ;  \gamma_2 \gamma_1 = \gamma_1 \gamma_2, 
( \gamma_2 \gamma_3)^2=( \gamma_3 \gamma_2)^2,  (\gamma_3 \gamma_1)^2 =( \gamma_1 \gamma_3)^2
\right> .
\label{fundamentalgroup(1,0)(0,1)(2,2)}
\end{equation}
 Here $\gamma_1$ (resp. $\gamma_2$) denotes the loop around $x=0$ (resp. $y=0$), while $\gamma_3$ denotes
the loop around  the parabola   as drawn in  Figure \ref{fig:fundamentalgroup } (precise parametrisation of loops is available in \cite{Kaneko}).
The loop around the line at infinity $: \P^2 \setminus \C^2$ is represented by $(\gamma_1\gamma_3 \gamma_2\gamma_3 )^{-1}.$

\section{ Monodromy calculus by Mellin-Barnes integrals}

To obtain a monodromy representation of the solution space to  
the system (\ref{horn(1,0)(0,1)(2,2)})
we try to use the following
Mellin-Barnes integrals that span a 4-dimensional  solution space to it,
\begin{equation}
 f^{k}_{i,j}(x,y)= \int_{\Gamma_k }\frac{\Gamma(s)^{2-i}\Gamma(t)^{2-j}\Gamma(1-2(s+t))}{\Gamma(1-s)^i\Gamma(1-t)^j}x^{-s} y^{-t} e^{-(s i +tj) \pi {\sqrt-1}  }ds \wedge dt,
\label{ex(1,0)(0,1)(2,2)}
\end{equation}
where $0 \leq i \leq 1$, $0 \leq j \leq 1$. 

Especially 
we have the following holomorphic solution in the neighbourhood of $(x,y) = (0,0)$ 
\begin{equation}
f^1_{1,1} (x,y)  = \sum_{(i_1, i_2 ) \in \Z^2_{\geq 0}} \frac{(2i_1 + 2i_2)!}{(i_1!)^2 (i_2!)^2} x^{i_1} y^{i_2}.
\label{holsolution}
\end{equation}

Let us denote by $H$ the image of the homomorphism
$$ \rho: \pi_1(\C^2 \setminus  S) \longrightarrow   GL(4, \C)$$
induced by the monodromy action along loops on the base solution vector 
$\vec f = (f_{00}, f_{10}, f_{01}, f_{11})$ defined by (\ref{ex(1,0)(0,1)(2,2)}).

To characterise the domain of convergence of $ f^{k}_{i,j}(x,y)$ we recall the notion of amoeba.
\begin{definition}
\label{def:amoeba}
\rm
The {\it amoeba}~$\mathcal{A}_\phi$ of a polynomial~$\phi(x,y)$
(or of the algebraic hypersurface $\{(x,y) \in  (\C^\ast)^2 ; \phi(x,y)=0 \}$) is defined to be the
image of the hypersurface~$\phi^{-1}(0) $ under the map ${\rm Log } :
(x,y)\mapsto (\log |x|,\log |y|) \in \R^2.$
\end{definition}
Let~$\mathcal{A}(\phi)$ denote the amoeba of the singularity of
the hypergeometric system  (\ref{horn(1,0)(0,1)(2,2)}) with $\phi(x,y) =  (16(x-y)^2 -8(x+y) +1). $
The complement to the amoeba $\mathcal{A}(\phi)$ consists of three connected components $M_1, M_2, M_3$
such that $$ u_k^--C^{\vee}_k  \subset M_k \subset u_k^+- C^{\vee}_k,\;\;\;$$
for some $u_k^- \in M_k$  $u_k^+ \in Log (\mathcal{A}_\phi),  k=1,2,3$ (see Two- sided Abel lemma \cite[Lemma11  ]{PST}). Here $  C^{\vee}_k$ is the dual cone to the cone $ C_k$ 
defined by replacing $\Z$ by $\R$ in the definition (\ref{poles}) of $\Gamma_k,$ $k=1,2,3$.
After \cite[Theorem 5.3]{SadykovTanabe} the convergence domain of  $ f^{k}_{i,j}(x,y)$ contains $Log^{-1}( M_k)$ 
for every fixed $k \in \{1,2,3\}$
and for all $0 \leq i \leq 1$, $0 \leq j \leq 1$.

\begin{proposition}
The analytic continuation of 4 linearly independent solutions (\ref{ex(1,0)(0,1)(2,2)}) to 
the Horn hypergeometric system (\ref{horn(1,0)(0,1)(2,2)}) gives the following monodromy representation $H_0 =<M_{10}, M_{20}, M_{1 \infty}, M_{2\infty}> \leq H$ (a proper subgroup), 
$$
M_{10}=
\left(
\begin{array}{cccc}
 1 & -2\pi i &0 & 0      \\
 0 & 1 & 0      &  0 \\
 0 & 0 & 1      & -2\pi i      \\
 0 & 0 & 0      & 1
\end{array}
\right), \quad
M_{20}=
\left(
\begin{array}{cccc}
 1 & 0 & -2\pi i       & 0      \\
 0 & 1      & 0      &  -2\pi i      \\
 0 & 0      & 1      & 0  \\
 0 & 0      & 0      & 1
\end{array}
\right),
$$

$$
M_{1\infty}^{-1}=
\left(
\begin{array}{cccc}
 1 & 0 &2\pi i  & 0      \\
 0 & -1 & -2      &  0 \\
 0 & 0 & 1      & -2\pi i      \\
 0 & 0 & 0      & -1
\end{array}
\right), \quad
M_{2\infty}^{-1}=
\left(
\begin{array}{cccc}
 1 & 2\pi i &0 & 0      \\
 0 & 1 & 0      &  -2\pi i  \\
 0 & -2 & -1      & 0     \\
 0 & 0 & 0      & -1
\end{array}
\right). \quad
$$

\label{rank4monodromy}
\end{proposition}

Here the local monodromy matrices act on the solution space from right. That is to say for ${\vec a} \in \C^4 \cong <{\vec a}, \vec f> $ with $\vec f = (f_{00}, f_{10}, f_{01}, f_{11}),$
the monodromy action around $x=0$ is given by $\vec a \rightarrow {\vec a} M_{10}.$ 
The local monodromy acts on the column vector of solutions   $^t \vec f$ from left.

\begin{proof}
We shall use a method  (named  Mellin-Barnes contour throw \cite[ Proposition 6.6]{SadykovTanabe}) to find analytic continuation of  an integral (\ref{ex(1,0)(0,1)(2,2)}) from one domain of convergence to another. This is a generalisation of  a method to calculate connection matrix 
 for the univariate hypergeometric function by means of Barnes integrals (\cite[2.4.6]{IKSY}, \cite{Smith} ).

Let us denote by $\lambda_{10}(f^k_{i,j})$ the result of the monodromy action on $f^k_{i,j}(x,y)$ around $x=0$
$$ \lambda_{10}(f^k_{i,j}) (x,y) = f^k_{i,j}(e^{2 \pi {\sqrt-1}  }x,y), \;\;\;  0 \leq i \leq 1,   0 \leq j \leq 1.$$
 In an analogous way we denote
$$   \lambda_{20}(f^k_{i,j}) (x,y) = f^k_{i,j}(x, e^{2 \pi {\sqrt-1}  }y), \;\;\;  0 \leq i \leq 1,   0 \leq j \leq 1. $$

For $f^k_{i,j}(x,y)$ convergent in the neighbourhood of $(x,y) = (\infty,0)$  the result of the clockwise monodromy action on it around $x=\infty$ is denoted by
$$   \lambda_{1 \infty}(f^k_{i,j}) (x,y) = f^k_{i,j}( e^{2 \pi {\sqrt-1}  }x, y), \;\;\;  0 \leq i \leq 1,   0 \leq j \leq 1. $$

For $f^\ell_{i,j}(x,y)$ convergent in the neighbourhood of $(x,y) = (0, \infty)$  clockwise turn around $y = \infty$ yields
$$   \lambda_{2 \infty}(f^\ell_{i,j}) (x,y) = f^\ell_{i,j}(x,  e^{2 \pi {\sqrt-1}  } y), \;\;\;  0 \leq i \leq 1,   0 \leq j \leq 1. $$
Further we shall calculate the above monodromy actions on the local solutions.

\begin{itemize}
 \item The local monodromy of  $f^1_{1,1}(x,y)$ around $x=0.$
\end{itemize}

The residue 
$$\sum_{n\geq 0, m\geq 0 } Res_{s=-n} Res_{t=-m}
\frac{\Gamma(s)\Gamma(t)\Gamma(1-2(s+t))}{\Gamma(1-s)\Gamma(1-t)}x^{-s} y^{-t} e^{-(s  +t) \pi {\sqrt-1}  }$$
will give us 
a function (\ref{holsolution}) holomorphic near $(0,0)$ and in $Log^{-1}(M_1)$. Thus   $\lambda_{1 0}(f^1_{1,1}) (x,y) = f^1_{1,1} (x,y), $
$\lambda_{2 0}(f^1_{1,1}) (x,y) = f^1_{1,1} (x,y). $

\begin{itemize}
 \item The local monodromy of  $f^1_{0,1}(x,y)$ around $x=0.$
\end{itemize}

The residue 
$$\sum_{n\geq 0, m\geq 0 } Res_{s=-n} Res_{t=-m}
\frac{\Gamma(s)^2\Gamma(t)\Gamma(1-2(s+t))}{\Gamma(1-t)} ( e^{2 \pi {\sqrt-1}}x)^{-s} y^{-t} e^{-t \pi {\sqrt-1}  }$$
turns out to be
$$\sum_{n\geq 0, m\geq 0 } Res_{s=-n} Res_{t=-m}
\frac{\Gamma(s)^2\Gamma(t)\Gamma(1-2(s+t))}{\Gamma(1-t)} x^{-s} y^{-t} e^{- t \pi {\sqrt-1}  }$$
$$ -2 \pi {\sqrt-1} \sum_{n\geq 0, m\geq 0 } Res_{s=-n} Res_{t=-m}
\frac{\Gamma(s)\Gamma(t)\Gamma(1-2(s+t))}{\Gamma(1-s)\Gamma(1-t)}x^{-s} y^{-t} e^{-(s  +t) \pi {\sqrt-1}  },$$
i.e. $  \lambda_{1 0}(f^1_{0,1}) (x,y) = f^1_{0,1} (x,y) - {2  \pi {\sqrt-1}  }  f^1_{1,1} (x,y).$

\begin{itemize}
 \item The local monodromy of  $f^1_{1,0}(x,y)$ around $x=0.$
\end{itemize} 
The residue 
$$\sum_{n\geq 0, m\geq 0 } Res_{s=-n} Res_{t=-m}
\frac{\Gamma(s)\Gamma(t)^2\Gamma(1-2(s+t))}{\Gamma(1-t)} (e^{2 \pi {\sqrt-1}} x)^{-s} y^{-t} e^{-s \pi {\sqrt-1}  }$$
equals to $f^1_{1,0}(x,y)$ itself i.e. $  \lambda_{1 0}(f^1_{1,0}) (x,y) =  f^1_{1,0}(x,y).$

 \begin{itemize}
 \item The local monodromy of  $f^1_{0,0}(x,y)$ around $x=0.$
\end{itemize} 
The residue 
$$\sum_{n\geq 0, m\geq 0 } Res_{s=-n} Res_{t=-m} \Gamma(s)^2\Gamma(t)^2\Gamma(1-2(s+t)) (e^{2 \pi {\sqrt-1}} x)^{-s} y^{-t}$$
turns out to be
$$\sum_{n\geq 0, m\geq 0 } Res_{s=-n} Res_{t=-m} \Gamma(s)^2\Gamma(t)^2\Gamma(1-2(s+t))  x^{-s} y^{-t}$$
$$  -2 \pi {\sqrt-1}\sum_{n\geq 0, m\geq 0 } Res_{s=-n} Res_{t=-m} \frac{\Gamma(s)\Gamma(t)^2\Gamma(1-2(s+t))}{\Gamma(1-s)} x^{-s} y^{-t} e^{-s \pi {\sqrt-1}  }$$
i.e. 
$  \lambda_{1 0}(f^1_{0,0}) (x,y) = f^1_{0,0} (x,y) - {2  \pi {\sqrt-1}  }  f^1_{1,0} (x,y).$

 \begin{itemize}
 \item As for the local monodromy of  $f^1_{i,j}(x,y)$, $ 0 \leq i \leq 1,   0 \leq j \leq 1$ around $y=0$ the calculation is symmetric with respect to the exchange of variables $x$ and $y.$ 
\end{itemize} 
$$  \lambda_{2 0}(f^1_{0,0}) (x,y) = f^1_{0,0} (x,y) - {2  \pi {\sqrt-1}  }  f^1_{0,1} (x,y), \lambda_{2 0}(f^1_{1,0}) (x,y) = f^1_{1,0} (x,y) - {2  \pi {\sqrt-1}  }  f^1_{1,1} (x,y)$$
 $$  \lambda_{2 0}(f^1_{0,1}) (x,y) =  f^1_{0,1}(x,y),   \lambda_{2 0}(f^1_{1,1}) (x,y) =  f^1_{1,1}(x,y).$$

\begin{itemize}
 \item The local monodromy of  $f^2_{\ast,\ast}(x,y)$ around $y=0$ gives the same result as the local monodromy of  $f^1_{\ast,\ast}(x,y)$   around $y=0$.
\end{itemize} 

\begin{itemize}
 \item The local monodromy of  $f^2_{\ast,\ast}(x,y)$ induced by a clockwise turn around $\frac{1}{x}=0,\frac{1}{x} \mapsto  \frac{e^{-2 \pi {\sqrt-1}}  }{x}.$  
We have the development 
$$ f^2_{1,0} (x,y) = -\frac{1}{2 x} +\frac{\frac{1}{2} i
   (\log (x)+i \pi )-\frac{1}{2} i \log (y)-i \gamma -i \psi
   \left(\frac{1}{2}\right)}{\sqrt{x}}+\frac{-6 y-1}{12 x^2}+  \frac{-30 y^2-20 y-1}{60 x^3}+...$$
($\gamma=$Euler constant, $\psi (z) =\Gamma'(z)/ \Gamma(z)$) that gives us 
$$\lambda_{1 \infty}(f^2_{1,0}) (x,y) +  f^2_{1,0} (x,y)=-\frac{-30 \pi  x^{5/2}+30 x^2+30 x y+5 x+30 y^2+20 y+1}{30 x^3}+ ...$$
We compare it with the development  
$$ f^2_{0,1} (x,y)= \frac{1+5 x+20 y+30 x^2+30 x y+30 y^2-30 \pi  x^{5/2}}{60 x^3}+...$$
and conclude 
$$ \lambda_{1 \infty}(f^2_{1,0}) (x,y) =-  f^2_{1,0} (x,y) - 2  f^2_{0,1} (x,y)$$
Similar residue calculus gives us the following results.
\end{itemize} 
$$  \lambda_{1\infty}(f^2_{0,0}) (x,y) = f^2_{0,0} (x,y) +  {2  \pi {\sqrt-1}  }  f^2_{0,1} (x,y), $$
 $$  \lambda_{1 \infty}(f^2_{0,1}) (x,y) =  f^2_{0,1}(x,y)  - {2  \pi {\sqrt-1}  }  f^2_{1,1} (x,y)  ,  \lambda_{1 \infty}(f^2_{1,1}) (x,y) = - f^2_{1,1}(x,y).$$

\begin{itemize}
 \item The local monodromy of  $f^3_{\ast,\ast}(x,y)$  induced by a clockwise turn around $\frac{1}{y}=0,\frac{1}{y} \mapsto  \frac{e^{-2 \pi {\sqrt-1}}  }{y}$ the result can be obtained from the  the calculation of $\lambda_{1 \infty} (f^2_{\ast,\ast}(x,y))$ due to a symmetry with respect to the exchange of variables $x$ and $y.$    
\end{itemize} 
$$  \lambda_{2 \infty}(f^3_{0,0}) (x,y) = f^3_{0,0} (x,y) +  {2  \pi {\sqrt-1}  }  f^3_{1,0} (x,y), \lambda_{2 \infty}(f^3_{1,0}) (x,y) =  f^3_{1,0} (x,y) -  {2  \pi {\sqrt-1}  }f^3_{1,1} (x,y)$$
 $$    \lambda_{2 \infty}(f^3_{0,1}) (x,y) = -  f^3_{0,1}(x,y)  - 2  f^3_{1,0} (x,y)  ,    \lambda_{2 \infty}(f^3_{1,1}) (x,y) = - f^3_{1,1}(x,y).$$

\begin{itemize}
 \item The local monodromy of  $f^3_{\ast,\ast}(x,y)$ around $x=0$ gives the same result as the local monodromy of  $f^1_{\ast,\ast}(x,y)$   around $x=0$.
\end{itemize}

We shall remark here  that the Mellin-Barnes contour throw sends
$f^k_{i,j}(x,y)$ (residues at poles in $\Gamma_k$, holomorphic in $Log^{-1}(M_k)$) to $f^\ell_{i,j}(x,y)$
 (residues at poles in $\Gamma_\ell$,  holomorphic in $Log^{-1}(M_\ell)$) for every  $0 \leq i \leq 1,   0 \leq j \leq 1. $
Thus we have no need to calculate the connection matrix like in \cite[2.4.6]{IKSY}, \cite{Smith}  if we choose the solution basis 
(\ref{ex(1,0)(0,1)(2,2)}).

In the following figure the analytic continuation between the residues along $\Gamma_1$ and those along $\Gamma_2$
is illustrated. By the same principle we can calculate the analytic continuation between residues  $\Gamma_k$ and  $\Gamma_\ell$
for every $\{k,\ell \} \subset \{1,2,3\}.$
\begin{center}
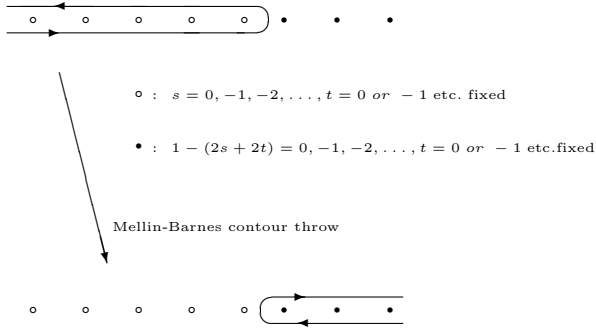
\begin{figure}[ht!]
\vbox{
\begin{minipage}{3cm}
\begin{picture}(150,150)
  \put(0,125){\vector(1,0){20}}
  \put(20,125){\line(1,0){70}}
  \put(67,125){\line(1,0){6}}
  \put(87,125){\line(1,0){8}}
   \put(105,130){\circle*{2}}
   \put(125,130){\circle*{2}}
   \put(145,130){\circle*{2}}
\put(95,130){\oval(8,10)[r]}
   \put(10,130){\circle{2}}
   \put(30,130){\circle{2}}
   \put(50,130){\circle{2}}
   \put(70,130){\circle{2}}
   \put(90,130){\circle{2}}
%
  \put(95,135){\line(-1,0){90}}
  \put(73,135){\line(-1,0){6}}
  \put(53,135){\line(-1,0){6}}
  \put(33,135){\vector(-1,0){15}}
  \put(20,135){\line(-1,0){20}}
%
   \put(105,20){\circle*{2}}
   \put(125,20){\circle*{2}}
   \put(145,20){\circle*{2}}
   \put(10,20){\circle{2}}
   \put(30,20){\circle{2}}
   \put(50,20){\circle{2}}
   \put(70,20){\circle{2}}
   \put(90,20){\circle{2}}
   \put(150,15){\vector(-1,0){40}}
   \put(100,25){\vector(1,0){13}}
   \put(100,15){\line(1,0){13}}
   \put(110,25){\line(1,0){40}}
\put(100,20){\oval(8,10)[l]}
   \put(50,102){\circle{2}}
   \put(55,100){\text{\tiny :\,\, $s = 0,-1,-2,\ldots , t=0 \;or\; -1\; {\rm etc.\; fixed  } $}}
   \put(50,82){\circle*{2}}
   \put(55,80){\text{\tiny :\,\, $1-(2s+2t) = 0,-1,-2,\ldots, t=0 \;or\; -1\; {\rm etc. fixed}$}}
   \put(20,110){\vector(1,-4){18}}
   \put(40,50){\text{\tiny Mellin-Barnes contour throw}}
\end{picture}
\end{minipage}
} 
\caption{Mellin-Barnes contour throw}
\label{fig:Mellin-Barnes
contour throw}
\end{figure}
\end{center}

In conclusion we obtained the matrices
$M_{10}, M_{20}, M_{1\infty}, M_{2\infty}.$ In fact the calculation of  $M_{10}, M_{20}$ can be done with the aid of local monodromy around $x=0$ of solutions to  Pochhammer hypergeometric equation
\begin{equation}
\begin{array}{l}
\left(\theta_{x}^n  - x (n\theta_x +1)\cdots (n\theta_x +n)\right)f(x)=0\\
\end{array}
\label{equnivariate1}
\end{equation}
for $n=2.$
See Appendix, Lemma 
\ref{lemmaunivariate}.
Thus the essential calculus is reduced to that of $M_{1\infty}$ as we see 
$$ M_{2\infty} =E_{2,3} M_{1\infty} E_{2,3}$$
for
$$ E_{2,3}= \left(
\begin{array}{cccc}
 1 & 0 & 0 & 0 \\
 0 & 0 & 1 & 0 \\
 0 & 1 & 0 & 0 \\
 0 & 0 & 0 & 1
\end{array}
\right)$$
that arises because of a symmetry between $x$ and $y$ variables.

According to the presentation (\ref{fundamentalgroup(1,0)(0,1)(2,2)}) this method allows us to calculate at most  the monodromy representation of the group $< \gamma_1, \gamma_2, \gamma_1\gamma_3 \gamma_2\gamma_3 > $ that is a proper subgroup of $  \pi_1 (\C^2 \setminus S ). $
Therefore the group $H_0$ generated by above 4 generators is a proper subgroup of $H.$ 
\end{proof}

This way to consider the analytic continuation by means of Mellin-Barnes contour throw has been used to prove the key statement 
Proposition 6.6 in \cite{SadykovTanabe}.

\begin{remark}
{\rm From this proposition we see easily that this monodromy representation has a  1-dimensional invariant subspace
$<(0,0,0,1)>$ (corresponding to the solution space spanned by $f_{11}$ : a solution holomorphic at $(x,y)=(0,0)$) 
and a 2-dimensional (resp. 3-dimensional)  invariant subspace $<(0,1,1, 0),$ $(0,$ $0,$ $0,$ $1)>$ (resp. $<(0,1,1,$  $ 0), $  $ (0,1,$ $-1, 0),$ $  (0,0,0,1)>.$

This representation has no 2-dimensional subspace with irreducible monodromy action.
Even though the 2-dimensional  solution space spanned by $f_{10}+f_{01}, f_{00}$
corresponds to the space of period integrals of an elliptic curve $\bar {\cal X}$
in $\P^1 \times \P^1$ (whose affine part $\bar {\cal X} \cap (\C^\ast)^2$ is  isomorphic to  ${\cal X}^{(2,2)}_{x,y}$ for $(x,y) \in  (\C^\ast)^2 \setminus S $) its monodromy does not give rise to a group isomorphic to the principal subgroup of level $2$ :$\Gamma(2)$
as expected. More precisely, the base change by 
$$L=\left(
\begin{array}{cccc}
 2 i \pi  & 0 & 0 & 0 \\
 0 & -1+2 i \pi  & 1 & -2 i \pi  \\
 0 & 1 & -1 & 2 i \pi  \\
 0 & -1 & 1 & 0
\end{array}
\right)$$
yields a monodromy representation on a two dimensional solution subspace $V$
such that 
$$ M^L_{10}\mid_V = M^L_{20}\mid_V = \left(
\begin{array}{cc}
 2 & -1 \\
 1 & 0
\end{array}
\right),  M^L_{10} (M^L_{1 \infty})^{-1}\mid_V = M^L_{20} (M^L_{2 \infty})^{-1}\mid_V = \left(
\begin{array}{cc}
 -1& 0\\
 0 & -1
\end{array}
\right) .$$
This monodromy representation is equivalent to
$$ <\left(
\begin{array}{cc}
 1 & 2 \\
 0 & 1
\end{array}
\right), \;\;  \left(
\begin{array}{cc}
 -1& 0\\
 0 & -1
\end{array}
\right)>  $$
i.e. a proper subgroup of $\Gamma(2).$
In other words  the monodromy representation $H_0$ gives only proper subgroup of full monodromy representation $H.$
The reason of this phenomenon lies in the fact that from the monodromy representation of Proposition \ref{rank4monodromy}
it is impossible to recover the monodromy action induced by the loop along $\gamma_3$ of  
(\ref{fundamentalgroup(1,0)(0,1)(2,2)}) i.e.  in this representation one of two Dehn twist actions around  cycles $\alpha, \beta $  (represented in Figure \ref{fig:pathpicture1}) is lacking. 
We may recover at our best the representation of $< \gamma_1, \gamma_2, \gamma_1\gamma_3 \gamma_2\gamma_3 > $ that is a proper subgroup of $  \pi_1 (\C^2 \setminus S ). $
To the moment we did not succeed to interpret Proposition  \ref{rank4monodromy} as a monodromy representation of the fundamental group (\ref{fundamentalgroup(1,0)(0,1)(2,2)}).}

\label{remark:gamma2}
\end{remark}

Here we remark the following facts:
$$rank\; ( M_{10}M_{1\infty} -Id_4) =2$$  not a pseudo-reflection 
 $$rank\;(( M_{10}M_{1\infty})^2- Id_4) =1$$
i.e. $ (M_{10}M_{1\infty})^2$ is a pseudo-reflection.

The following relations also hold,
$$M_{10} M_{20} = M_{20} M_{10}, $$
$$(M_{1\infty}^{-1} M_{10}^{-1} M_{20})^2= (M_{20} M_{1\infty}^{-1} M_{10}^{-1})^2 ,  $$
$$  (M_{1\infty}^{-1})^2= (M_{10} M_{1\infty}^{-1} M_{10}^{-1})^2.$$

We calculate the  Hermitian quadratic invariant $ {\sf H } :$ a $4 \times 4$  matrix
\begin{equation}
 ^t \bar{g} {\sf H } g =  {\sf H } ,
\label{hermitian}
\end{equation}
for every $g \in <M_{10}, M_{20}, M_{1\infty},  M_{2\infty}>:$ the monodromy representation $H_0$  of the system 
(~\ref{horn(1,0)(0,1)(2,2)}) as follows,
$$ {\sf H}= \left(
\begin{array}{cccc}
 0 & 0 & 0 & 0 \\
 0 & 0 &0  & 2 {\sqrt 2}\\
 0 & 0 & 0 & 2 {\sqrt 2} \\
 0 & 2 {\sqrt 2} & 2 {\sqrt 2} & 0
\end{array}
\right).
$$


Let $(\scEtilde_i)_{i=1}^4$ be the full strong exceptional collection on $D^b Coh\; (\P^1 \times \P^1 )$  given by
$$  (\scEtilde_1, \scEtilde_2 , \scEtilde_3, \scEtilde_4) = ({\scO}, {\scO}(1,0), {\scO}(1,1), {\scO}(2,1))$$
and $(\scFtilde_1, \scFtilde_2, \scFtilde_3, \scFtilde_4)$ be
its right dual exceptional collection
characterised by the condition
$$
 Ext^k(\scEtilde_{5-i}, \scFtilde_j) = 
  \begin{cases}
   \C & i=j, \text{ and } k=0, \\
   0 & \text{otherwise}.
  \end{cases}
$$The Euler form on the Grothendieck group $K(\P^1 \times \P^1)$
defined by 
$$\chi ( {\mathcal E}, {\mathcal F})  = \sum_{n\geq 0} (-1)^n dim Ext^n( {\mathcal E}, {\mathcal F}  ).$$
is neither symmetric nor anti-symmetric,
whereas that on $K(\mathcal Y)$ is anti-symmetric.

The bases $\{ [\scEtilde_i] \}_{i=1}^4$
and $\{ [\scFtilde_i] \}_{i=1}^4$ of $K(\P^1 \times \P^1)$
are dual to each other in the sense that
\begin{align} \label{eq:Euler_form}
 \chi(\scEtilde_{5-i}, \scFtilde_j) = \delta_{ij}.
\end{align}
We will write the derived restrictions
of $\scEtilde_i$ and $\scFtilde_i$
to $\cal Y$ as
$\scEbar_i$ and $\scFbar_i$ respectively. After\cite[Lemma 5.4]{Seidel} the split generator on the  curve $\cal Y$ with bidegree $(2, 2)$  can be obtained by restricting the full exceptional collection 
to $\cal Y$.
Unlike $\{ [\scEtilde_i] \}_{i=1}^4$
and $\{ [\scFtilde_i] \}_{i=1}^4$,
$\{ [\scEbar_i] \}_{i=1}^4$ and
$\{ [\scFbar_i] \}_{i=1}^4$ are not bases of $K(\cal Y)$,
and their images in the numerical Grothendieck group
are linearly dependent.

The Gram matrix $\sf G$  with respect to  the Euler form 
      of the split generator  $\{ [\scFbar_i] \}_{i=1}^4$ is calculated as follows. 
\begin{equation}
 {\sf G}= \left( \chi([\scFbar_i], [\scFbar_j])\right)_{i,j =1}^4= \left(
\begin{array}{cccc}
 0 & -2 & 0 & 2 \\
 2 & 0 & -2 & 0 \\
 0 & 2 & 0 & -2 \\
 -2 & 0 & 2 & 0
\end{array}
\right)
\label{gram}
\end{equation}

\begin{proof}
As the Euler form for the restricted sheaves  $\{ \scFbar_i \}_{i=1}^4$ satisfies
\begin{equation} 
\chi ( [\scFbar_{i_1}], [ \scFbar_{i_2} ] ) =    \chi ( \scF_{i_1},  \scF_{i_2}  ) -  \chi ( \scF_{i_2},  \scF_{i_1}  ), 
\label{antisymmetric}
\end{equation} and    $ \chi ( \scF_{i_1},  \scF_{i_2}  )=0$ if $i_1 > i_2$, the Gram matrix must be anti-symmetric. 

From \cite{Beil} it follows  that 
$$  \chi ( {\mathcal O},  {\mathcal O}) =   \chi ( {\mathcal O}(1,0),  {\mathcal O}(1,0))=   \chi ( {\mathcal O}(1,1),  {\mathcal O}(1,1))=  \chi ( {\mathcal O}(2,1),  {\mathcal O}(2,1))=  1, $$
$$  \chi ( {\mathcal O},  {\mathcal O}(1,0)) =   \chi ( {\mathcal O}(1,0),  {\mathcal O}(1,1))=   \chi ( {\mathcal O}(1,1),  {\mathcal O}(2,1))=2, $$
$$   \chi ( {\mathcal O} ,  {\mathcal O}(1,1)) =   \chi ( {\mathcal O}(1,0),  {\mathcal O}(2,1))=4, $$
$$   \chi ( {\mathcal O} ,  {\mathcal O}(2,1)) = 6.$$

These relations  (i.e. $\left( \chi([\scE_i], [\scE_j])\right)_{i,j =1}^4$)  entail
$$  \left( \chi(\scF_i, \scF_j )\right)_{i,j =1}^4 
=   \left(
\begin{array}{cccc}
 1 & -2 & 0 & 2 \\
 0 & 1 & -2 & 0 \\
 0 & 0 & 1 & -2 \\
 0 & 0 & 0 & 1
\end{array}
\right)
= \left(
\begin{array}{cccc}
 1 & 2 & 4 & 6 \\
 0 & 1 & 2 & 4 \\
 0 & 0 & 1 & 2 \\
 0 & 0 & 0 & 1
\end{array}
\right)^{-1}.
$$
This upper triangle matrix
together with (\ref{antisymmetric}) calculates the Gram matrix   ${\sf G}$
(\ref{gram}).
\end{proof}

\begin{proposition}
We can choose an unitary base change matrix $R$
$${\sf  R}=\frac{1}{2}\left(
\begin{array}{cccc}
 \sqrt{2} & 0 & \sqrt{2} & 0 \\
 -i & -1 & i & -1 \\
 -i & 1 & i & 1 \\
 0 & -\sqrt{2} & 0 & \sqrt{2}
\end{array}
\right), \;\;\;  ^t \bar{\sf R} {\sf R} = Id_4  $$
 such that 
$$ {\sf H^ R} =  {\sqrt -1}{\sf G}$$
for the Hermitian quadratic invariant  ${\sf H}$ (\ref{hermitian})  of the monodromy subgroup $H_0.$
\end{proposition}


In fact by a  direct calculation we see that 
${\sqrt -1}{\sf G}$ is an element of a one dimensional real  vector space of Hermitian quadratic invariants of  
$$  <M_{10}^{\sf R},  M_{20}^{\sf R},   M_{1\infty}^{\sf R},    M_{2\infty}^{\sf R}  > \cong   H_0.$$


\section{ Monodromy calculus by generalised Picard-Lefschetz theorem.}

In \cite[Corollary 4.1, Remark 4.4]{GotoMatsumoto} (see also \cite{Kaneko} for generic parameter case) the following monodromy representation of the fundamental group 
(\ref{fundamentalgroup(1,0)(0,1)(2,2)})  with respect to a certain twisted cycle basis has been obtained by means of the generalised Picard-Lefschetz theorem.
A solution holomorphic in the neighbourhood of $(x,y)=(0,0)$ can be written down in the form (\ref{abholsolution}).

\begin{proposition}
The solution to the Horn hypergeometric system (\ref{horn(1,0)(0,1)(2,2ab)}) with rank 4 admits the following monodromy representation
including the cases with  $a,b \in \Z;$ 

$$
\rho_{a,b}(\gamma_1)=\left(
\begin{array}{cccc}
 1 & 0 & 0 & 0 \\
 1 &  e^{2 i a \pi }  & 0 & 0 \\
 0 & 0 & 1 & 0 \\
 0& 0& 1&  e^{2 i a \pi }
\end{array}
\right)
$$

$$
\rho_{a,b}(\gamma_2)=\left(
\begin{array}{cccc}
 1 & 0 & 0 & 0 \\
 0& 1 & 0 & 0 \\
 1 & 0 &  e^{2 i b \pi } & 0 \\
 0& 1& 0&   e^{2 i b \pi }
\end{array}
\right),$$

$$
\rho_{a,b}(\gamma_3)=\left(
\begin{array}{cccc}
1 & -1-e^{2 i a \pi } & -1-e^{2 i b \pi } & 1-e^{2 i (a+b) \pi } \\
 0 & 1 & 0 & 0 \\
 0 & 0 & 1 & 0 \\
 0 & 0 & 0 & 1
\end{array}
\right).
$$
\label{picardlefschetz}
\end{proposition}

The  Hermitian quadratic invariant (unique up to a real constant multiplication)  associated to the Appell's system $F_4$
(\ref{horn(1,0)(0,1)(2,2ab)}) can be calculated as
\begin{equation}
\tilde {\sf H} = 
\left(
\begin{array}{cccc}
{\sf V_1}\\
{\sf V_2} \\
{\sf V_3} \\
{\sf V_4} 
\end{array}
\right)
\label{Kanekohermitian}
\end{equation}

$$ {\sf V_1}= (0 , i \left(1+e^{2 i a \pi }\right)  , i \left(1+e^{2 i b \pi }\right) , i \left(-1+e^{2 i (a+b) \pi }\right) )$$
$$ {\sf V_2}= (-i \left(1+e^{-2 i a \pi }\right) , -i e^{-2 i a \pi } \left(-1+e^{4 i a \pi }\right) , i \left(e^{-2 i a \pi
   }-e^{2 i b \pi }\right) , -i e^{2 i b \pi } \left(-1+e^{2 i a \pi }\right) \left(1-e^{-2 i (a+b) \pi }\right) ) $$
$$ {\sf V_3}=-i  ( \left(1+e^{-2 i b \pi }\right) ,  \left(e^{2 i a \pi }-e^{-2 i b \pi }\right) ,  e^{-2 i b \pi }
   \left(-1+e^{4 i b \pi }\right) ,  e^{2 i a \pi } \left(-1+e^{2 i b \pi }\right) \left(1-e^{-2 i (a+b) \pi }\right))$$

$$ {\sf V_4}=\left(1-e^{-2 i (a+b) \pi }\right) (  i  , i \left(-1+e^{2 i a \pi }\right)  , i
   \left(-1+e^{2 i b \pi }\right),  i \left(-1+e^{2 i a \pi }\right)   \left(-1+e^{2 i b \pi }\right) ).$$

The analytic variety in the space of $4 \times 4$ Hermitian matrices represented by (\ref{Kanekohermitian}) depending
on parameters $a,b$
form a closed set.
Thus we can consider the limit case $a, b\rightarrow 0$ and obtain  (after multiplication by $ \sqrt 2$)
\begin{equation}
\tilde {\sf H}_0=\left(
\begin{array}{cccc}
 0 & 2 i \sqrt{2} & 2 i \sqrt{2} & 0 \\
 -2 i \sqrt{2} & 0 & 0 & 0 \\
 -2 i \sqrt{2} & 0 & 0 & 0 \\
 0 & 0 & 0 & 0
\end{array}
\right).
\label{kanekohermitian1}
\end{equation}

From the monodromy representation of
Proposition \ref{picardlefschetz} for the limit case $a, b\rightarrow 0$ we obtain

\begin{equation}
\rho_{(0,0)}(\gamma_1)= \left(
\begin{array}{cccc}
 1 & 0 & 0 & 0 \\
 1 &  1  & 0 & 0 \\
 0 & 0 & 1 & 0 \\
 0& 0& 1&  1
\end{array}
\right),
\; \rho_{(0,0)}(\gamma_2) = 
\left(
\begin{array}{cccc}
 1 & 0 & 0 & 0 \\
 0& 1 & 0 & 0 \\
 1 & 0 &  1 & 0 \\
 0& 1& 0& 1
\end{array}
\right),\;\;
\rho_{(0,0)}(\gamma_3)=\left(
\begin{array}{cccc}
 1 & -2 & -2 & 0 \\
 0 & 1 & 0 & 0 \\
 0 & 0 & 1 & 0 \\
 0 & 0 & 0 & 1
\end{array}
\right).
\label{kanekomonodromy0}
\end{equation}

\begin{proposition}
We can choose an unitary base change matrix $\sf R_0$ 
$$
{\sf R}_0=\frac{1}{\sqrt 2} \left(
\begin{array}{cccc}
 i & 0 & -i & 0 \\
 0 & -\frac{1+i}{\sqrt{2}} & 0 & -\frac{1-i}{\sqrt{2}} \\
 0 & \frac{1-i}{\sqrt{2}} & 0 & \frac{1+i}{\sqrt{2}} \\
 1 & 0 & 1 & 0
\end{array}
\right), \;\; ^t{ \sf \bar R}_0. {\sf R}_0 = Id_4
$$
such that 
$$ \tilde{\sf H}_0^{\sf R_0} = {\sqrt -1} {\sf G}.$$
That is to say a pure imaginary multiple of the Gram matrix $\sf G$ (\ref{gram} ) spans a 1-dimensional real space of  Hermitian quadratic invariants of the monodromy representation  
of (\ref{horn(1,0)(0,1)(2,2)}) 
$$ <\rho_{(0,0)}(\gamma_1)^{\sf R_0}, \rho_{(0,0)}(\gamma_2)^{\sf R_0},  \rho_{(0,0)}(\gamma_3)^{\sf R_0} >$$
given by  (\ref{kanekomonodromy0}). 
\label{gotomatsumoto}
\end{proposition}

\begin{remark}

{\rm 1. There is no conjugation matrix that would send $M_{j,0}$ to $\rho_{(0,0)}(\gamma_j)$ for both $j=1, 2.$

2. The question about the faithfulness of the monodromy representation  (\ref{kanekomonodromy0}) deserves a special attention.
In other words, we ask whether the monodromy group given in  Proposition \ref{gotomatsumoto} is isomorphic to 
the fundamental group  $\pi_1(\C^2 \setminus S) $ given by
(\ref{fundamentalgroup(1,0)(0,1)(2,2)} ) . If the answer is negative  e.g.   (\ref{kanekomonodromy0})  gives rise to a subgroup 
strictly smaller than $\pi_1(\C^2 \setminus S) $, we may ask the same question about the monodromy representation  given in Proposition \ref{picardlefschetz} for generic values of $a,b$.}
 \end{remark}

\section{Appendix: Maximally unipotent local monodromy of the Poch-hammer hypergeometric equation. }

We prepare a lemma on the local monodromy around $x=0$
of the Pochhammer hypergeometric equation,
\begin{equation}
\begin{array}{l}
\left(\theta_{x}^n  - x (n\theta_x +1)\cdots (n\theta_x +n)\right)f(x)=0.\\
\end{array}
\label{equnivariate}
\end{equation}
with reducible monodromy  to which the Levelt type theorem \cite[Theorem 3.5]{BeukersHeckman}
cannot be directly applied. Despite the reducibility, in fact the Levelt type theorem holds \cite{Tan04},
 \cite[Theorem 1.1]{TaU13} in this case also.

\begin {lemma}
Let us consider a basis of the solution space to (~\ref{equnivariate})  as follows.
\begin{equation}
 f_{j}(x)= \sum_{k \in \Z_{\geq 0}}Res_{s=-k}\frac{\Gamma(s)^{n-j}\Gamma(1-ns)}{\Gamma(1-s)^j}x^{-s} e^{-\pi {\sqrt -1}js}ds,\; j=0, \cdots n-1.
\label{solunivariate}
\end{equation}
The monodromy around $x=0$ 
with respect to a basis (~\ref{solunivariate}) of solutions to (\ref{equnivariate})  is given as follows
$$ \left(
\begin{array}{cccccc}
 1 & -2\pi i  & 0& \cdots &0& 0     \\
 0 & 1  &-2\pi i &\cdots  &0& 0 \\
0 & 0& 1  & \cdots &0 & 0 \\
\vdots & \vdots  &\vdots &\ddots   &\vdots& \vdots \\
0& 0 & 0&\cdots & 1  &-2\pi i   \\
0& 0 & 0&  \cdots  & 0 & 1   
\end{array}
\right).$$
\label{lemmaunivariate}
 \end {lemma}
\begin{proof}
First of all we introduce a periodic meromorphic function 
$$ P(s) = \Gamma(s)\Gamma(1-s) e^{\pi {\sqrt -1}s} = \frac{\pi e^{\pi {\sqrt -1}s}}{ sin\; \pi s },$$
with period $1$ i.e. $P(s+1) =P(s)$
and a meromorphic function
$$H(s) = \frac{\Gamma(1-ns)}{\Gamma(1-s)^{n}}x^{-s} e^{-n\pi {\sqrt -1} s}.$$ 
This means that
$$ f_{n-k}(x) = \int_{\Gamma_0}P(s)^k H(s)ds$$ after the notation (~\ref{solunivariate}).
We shall denote by $\Gamma_0$ the integration contour turning counter clockwise around the negative integers points so that  the integration along it give the summation of residues   (~\ref{solunivariate}). 
Here we recall the partial fraction expansion of the cosecant function,
$$\frac{\pi}{ sin\; \pi s } = \frac{1}{s}+ \sum_{m \in \N} (-1)^m (\frac{1}{s+m} +\frac{1}{s-m}  )$$
It is clear that $P(s)$ has residue $(-1)^m$ at $s=-m \in \Z_{\leq 0}$ and there $P(s)^kH(s)$ has the only possible $k-$th order poles  on the negative real axis.  In summay the following relation would entail the desired result,
$$f_{n-k-1}(e^{2\pi {\sqrt -1}}x)= \int_{\Gamma_0}P(s)(P(s)^k H(s)e^{-2\pi {\sqrt -1}s})ds$$
 $$= \int_{\Gamma_0}P(s)(P(s)^k H(s))ds -2\pi {\sqrt -1}\int_{\Gamma_0}P(s)^k H(s)ds,$$
\begin{equation}
=f_{n-k-1}(x) -2\pi {\sqrt -1} f_{n-k}(x).
\label{logrecurrence}
\end{equation} 
for $k=1, \cdots, n-1$.

We show ( ~\ref{logrecurrence} )  by the following argument.
Let us introduce a function $$ B(s)=sP(s) = \sum_{m \geq 0} \frac{B_m (2 \pi {\sqrt -1} s)^m}{m!},$$ with $B_m:$Bernoulli number such that $B_0=1, B_1 = 1/2, B_{2m+1}=0.$
 The leading term of the asymptotic expansion at $x=0$ of a solution to (~\ref{equnivariate}) in the form of a linear combination of $(ln x)^{j},$ $j=0, \cdots, n-1$ completely determines how this solution is represented as a linear combination of solutions $f_{n-k-1}(x),$ $k=0, \cdots, n-1.$ This situation allows us to reduce the proof of 
(~\ref{logrecurrence}) to the following equality between residues.
$$Res_{s=0} ((\frac{B(s)}{s})^{k+1} H(s)e^{-2\pi {\sqrt -1}s})= 
Res_{s=0} ((\frac{B(s)}{s})^{k+1} H(s) - 2\pi {\sqrt -1} (\frac{B(s)}{s})^{k} H(s)).$$
The LHS of the above equality equals to
$$ \frac{1}{k!} (B(s)^{k+1}  H(s)e^{-2\pi {\sqrt -1}s})^{(k)}|_{s=0}$$
$$= \frac{1}{k!} \left((B(s)^{k+1}  H(s))^{(k)}|_{s=0}+ \sum_{\ell=0}^{k-1} \;_kC_\ell (B(s)^{k+1}  H(s))^{(\ell)}|_{s=0}(-2\pi {\sqrt -1} )^{k-\ell}  \right).$$Hence
$$ \frac{1}{k!} \left( (B(s)^{k+1}  H(s)e^{-2\pi {\sqrt -1}s})^{(k)}-  (B(s)^{k+1}  H(s))^{(k)}\right)|_{s=0} $$ 
$$= -2\pi {\sqrt -1} \left(    \sum_{\ell=1}^k    \frac{1}{(k-\ell)! \ell ! } (B(s)^{k+1}  H(s))^{(k-\ell)}|_{s=0} (-2\pi {\sqrt -1})^{\ell-1} \right),$$ as $B(0)=1.$
The required equality will be proven if 
$$ \sum_{\ell=1}^k    \frac{1}{(k-\ell)! \ell ! } (B(s)^{k+1}  H(s))^{(k-\ell)}|_{s=0} (-2\pi {\sqrt -1})^{\ell-1}  -  \frac{1}{(k-1)!} (B(s)^{k}  H(s))^{(k-1)}|_{s=0}$$
turns out to be zero. This difference is calculated 
as 
$$ \sum_{\ell=1}^k    \frac{1}{(k-\ell-1)!} ( \sum_{r=1}^{\ell +1} \frac{  (-2\pi {\sqrt -1})^{r-1} B^{(\ell -r+1)}(0) } {  (\ell-r+1)! r !  } ) (B(s)^{k}  H(s))^{(k-\ell-1)}|_{s=0} $$
$$ =\sum_{\ell=1}^k    \frac{ (-2\pi {\sqrt -1})^{\ell} }{(k-\ell-1)! \ell!} (B_\ell+ \sum_{r=0}^{\ell -1} \frac{  \ell ! B_r} {  (\ell-r+1)! r !  } -1 ) (B(s)^{k}  H(s))^{(k-\ell-1)}|_{s=0}. $$
The coefficient of the factor $(B(s)^{k}  H(s))^{(k-\ell-1)}|_{s=0}$ vanishes by virtue of the recurrent relation for Bernoulli numbers.

It is worthy noticing that the  equality
$$\int_{\Gamma_0}P(s)(P(s)^k H(s)e^{-2\pi {\sqrt -1}s})ds
 = \int_{\Gamma_0}P(s)(P(s)^k H(s))ds -2\pi {\sqrt -1}\int_{\Gamma_0}P(s)^k H(s)ds,$$
holds for any function $H(s)$ holomorphic in the neighbourhood of the negative real axis.
The last equality yields the desired result.
\end{proof}

\end{document}